\documentclass{article}

\usepackage{amssymb,amsmath, amsthm}
\usepackage{url} 
\usepackage[all]{xy}
\usepackage{graphicx} 

\newtheorem*{thm}{Theorem}

\newtheorem{prop}{Proposition}
\newtheorem{lem}{Lemma}
\newtheorem{rem}{Remark}
\newtheorem{clm}{Claim}

\title{Some examples of minimal Cantor set for IFSs with overlap}

\author{Katsutoshi Shinohara \thanks{FIRST, Aihara Innovative Mathematical Modelling Project,
JST, 4-6-1 Komaba, Meguro-ku, Tokyo 153-8505, Japan
Institute of Industrial Science, University of Tokyo, 4-6-1
Komaba, Meguro-ku, Tokyo 153-8505, Japan}}

\date{}

\begin{document}

\maketitle

\begin{abstract}
We give some examples of IFSs 
with overlap on the interval such that 
the semigroup action they give rise to has a minimal set homeomorphic to the Cantor set.

{ \smallskip

}
\end{abstract}

\section{Introduction}

In \cite{K}, the minimality of semigroup actions 
(equivalently, minimal set of iterated function systems, abbreviated as IFSs)
on the interval are discussed
and several examples of non-minimal actions are presented there.
For these examples, we certainly know 
that the action is not minimal but we could not decide the ``shape" of the 
minimal set (similar type of questions are extensively studied 
recently, see for example \cite{S} for complex dynamical systems case). 

In the case of group actions on the circle, there 
is a famous trichotomy of the minimal set (see for example Chapter~2 of \cite{N}): 
Either it is a finite set, equal to the whole manifold, or homeomorphic to the Cantor set. 
Thus once we know that the minimal set is niether the finite orbit nor the whole manifold, 
then we can immediately conclude that it is homeomorphic to the Cantor set.
On the other hand, in semigroup action case, because of the lack of homogeneity of the minimal set, 
this trichotomy is no longer valid. 
The best thing we know in general is, as is stated in \cite{BR} (see Theorem~5.2), 
the trichotomy of the following type: 
the minimal set is either a finite set, homeomorphic to the Cantor set or a closed set with non-empty interior. 
This trichotomy is not enough to conclude the type of minimal set 
just from the non-minimality of the action.

In this article, we give an example of class of IFSs 
with a minimal set homeomorphic to the Cantor set.
They are produced by performing some modifications 
on the example presented in \cite{K}. 
Thus, still I do not know the type of minimal set for the original examples,
see also Remark~\ref{r.open}.

After the announcement of this example, Masayuki Asaoka (Kyoto University) showed me 
another example of minimal Cantor set using the measure theoretic argument.  
In the appendix we present that example.

{\bf Acknowledgements} \quad
This research is supported by the Aihara Project, the FIRST program
from JSPS, initiated by CSTP. 
The author is thankful to Masayuki Asaoka, 
Kengo Shimomura,
Hiroki Sumi 
and 
Dmitry Turaev 
for helpful conversations.

\section{Axiomatic description of the example}

In this section, we prepare some definitions,
give the axiomatic description of the example and 
the precise statement of our main result.

\subsection{Overlapping}
We consider IFSs on the closed unit interval 
$I :=[0, 1]$
generated by two maps 
$f, g: I \to I$ satisfying the following conditions:
\begin{itemize}
\item $f, g:I \to I $ are $C^1$-diffeomorphisms on their images.
\item $f(0) = 0$ and $g(1) =1$.
\item $f(x) < x < g(x)$ for $x \in (0, 1)$.
\item $0 < g(0) < f(1) <1$.
\end{itemize}
We denote the set of the pairs of diffeomorphisms satisfying these conditions by
 $\mathcal{A} \subset (\mathrm{Diff}_{\mathrm{im}}^1(I))^2$ (by $\mathrm{Diff}_{\mathrm{im}}^1(I)$
we denote the set of $C^1$-maps from $I$ to itself which is diffeomorphism on its image).

We prepare some notations and basic definitions.
We denote the semigroup generated by $f$ and $g$ by $\langle f, g\rangle_{+}$. 
It acts on $I$ in a natural way.
For $x \in I$, the orbit of $x$, denoted by $\mathcal{O}_{+}(x)$, is defined to be the set 
$\{ \phi (x) \mid \phi \in \langle f, g\rangle_{+} \}$.
A non-empty set $K \subset I$ is called a \emph{minimal set} 
if for every $x \in K$, $\overline{\mathcal{O}_{+}(x)} = K$, 
where $\overline{X}$ denotes the closure of $X$.
Finally, we say that the action of 
$\langle f, g\rangle_{+}$ is minimal 
if $I$ is the minimal set, in other words, every point in $I$ has a dense orbit.

For IFSs generated by $ (f, g) \in\mathcal{A}$, we can prove the following (see Lemma~1 in \cite{K}): 
\begin{prop} 
There exists a unique minimal set $K$.
Furthermore, we have $K =\overline{\mathcal{O}_{+}(0)} =\overline{\mathcal{O}_{+}(1)}$.
\end{prop}
For $(f, g) \in \mathcal{A}$, 
put $W :=[g(0), f(1)] $ and call it \emph{overlapping region} ($W$ is the overlap between $f(I)$ and $g(I)$ ).
Note that by definition of $\mathcal{A}$, $W$ is an interval with non-empty interior.

\subsection{Alignment of fundamental domains}

We define two sequences of intervals $\{F_n \}$ and $\{G_n\}$
as follows:
\begin{itemize}
\item $F_{n} := [f^{n+1}(1), f^n(1) ]$.
\item $G_{n} := [g^n(0), g^{n+1}(0)]$.
\end{itemize}
Let us consider the following property for $(f, g)$:
\begin{center}
$W =f(I) \cap g(I)$ is contained in the interior of $F_1 \cup G_1$. 
\end{center}
We call this property \emph{single overlapping property} and denote it by 
$\mathsf{So}$. Note that this condition is equivalent to $f^2(1) < g(0)$ and $f(1) < g^2(0)$.
This implies that $\mathsf{So}$  is a $C^0$-open property in $\mathcal{A}$.

\subsection{Existence of the hole}
Suppose $(f, g) \in \mathcal{A}$ satisfies $\mathsf{So}$.
Consider the following condition (in the following, by $\mathrm{int}(X)$
we denote the set of the interior points of $X$):
\begin{center}
There are closed intervals 
$H_f \subset \mathrm{int}(F_1 \setminus W)$ and
$H_g \subset \mathrm{int}(G_1 \setminus W)$
with non-empty interior such that $g(H_f) = H_g$ and $f(H_g) = H_f$.
\end{center}
\begin{figure}
\centering
\includegraphics{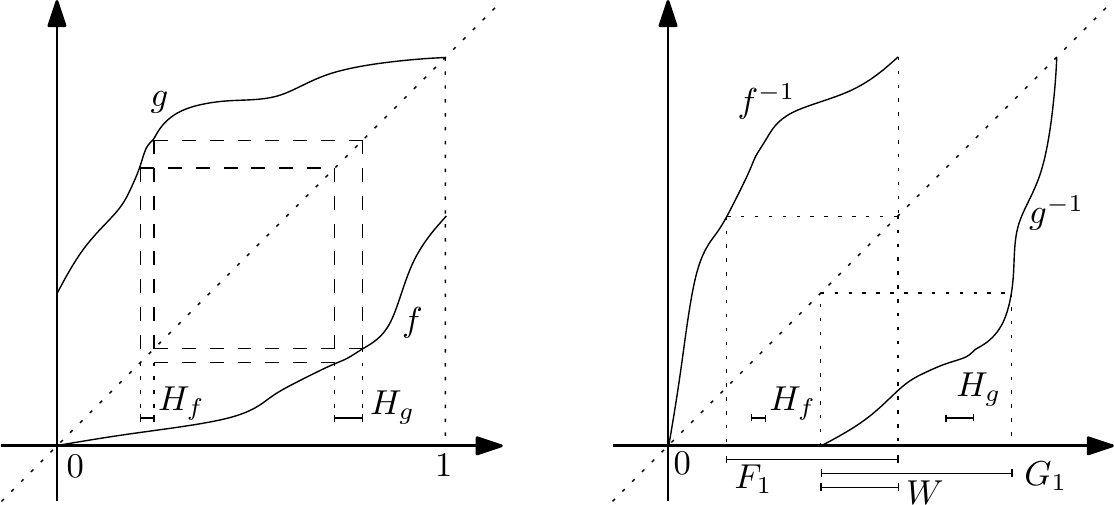}
 \caption{An example of IFS in $\mathcal{A}$ with $\mathsf{So}$ and $\mathsf{Ho}$.}
\label{fig.aliint}
\end{figure}
We call this property \emph{hole property} and denote it 
by $\mathsf{Ho}$ (see Figure~\ref{fig.aliint}).
If $(f, g) \in \mathcal{A}$ satisfies $\mathsf{So}$ and $\mathsf{Ho}$, then we have 
the following (see also Example~2 in \cite{K}):
\begin{prop}\label{p.tenganai}
If $(f, g) \in \mathcal{A}$ satisfies $\mathsf{So}$ and $\mathsf{Ho}$,
then the (unique) minimal set $K$ of $\langle f, g\rangle_{+}$ does not coincides with the whole interval. 
More precisely, we have $K \subset I \setminus \mathrm{int}(H_f \cup H_g)$.
\end{prop}
\begin{proof}
Suppose $K \cap  \mathrm{int}(H_f \cup H_g) \neq \emptyset$. 
We consider the case $ K \cap  \mathrm{int}(H_f) \neq \emptyset$
(The proof of the case $ K \cap  \mathrm{int}(H_g) \neq \emptyset$ is similar so we omit it).
Then we know that there exists a point $x_0 \in \mathcal{O}_{+}(0) \cap H_{f}$.
Since $0\not \in \mathrm{int}(H_f)$, 
it means that there exists a point $x_1 \in \mathcal{O}_{+}(0) $
such that $x_0=f(x_1)$ or $x_0=g(x_1)$ holds. However, 
since $\mathrm{int}(H_f) \cap g(I) = \emptyset$ by definition, 
the second option cannot happen. Hence we have $x_0 =g(x_1) \in \mathrm{int}(H_g)$
and it implies $x_1 \in \mathrm{int}(H_g)$. 
By repeating this process, we continue  taking the backward image of $x$ contained in 
$\mathrm{int}(H_f \cup H_g)$. 
At some moment it must be equal to $0$ but this is a contradiction.
\end{proof}

\subsection{Eventual expansion property}

We define two maps $\mathcal{F}: F_1 \setminus \{ f(1) \} \to G_1$ and 
$\mathcal{G}: G_1\setminus \{ g(0)\} \to F_1$ as follows:
For $x \in F_1 \setminus \{ f(1) \}$, 
let $n(x) $ be the least non-negative integer such that $g^{-n(x)}(f^{-1}(x)) \in G_1$ holds. 
This defines a piecewise $C^1$ map $\mathcal{F}: F_1 \setminus \{ f(1) \} \to G_1$
by setting $\mathcal{F}(x) := g^{-n(x)}(f^{-1}(x))$ 
($\mathcal{F}$ is just an ``induced map" from $F_1$ to $G_1$ of the inverse system $(f^{-1}, g^{-1})$).
We define $\mathcal{G}$ similarly, exchanging the role of $f$ and $g$.

Then let us consider the following property: 
\begin{center}
$\mathcal{F}$ is uniformly expanding outside $H_f$ and
$\mathcal{G}$ is as such outside $H_g$.
\end{center}
More precisely,
there exists $\mu > 1$ such that at each $y \in F_1 \setminus H_f$ (resp. $y \in G_1 \setminus H_g$), 
$\mathcal{F}'(y) > \mu$ (resp. $\mathcal{G}'(y) > \mu$)
 where $y$ ranges over the points for which the term $\mathcal{F}'(y)$
(resp. $\mathcal{G}'(y)$) makes sense. 
We call this property \emph{eventual expansion property} and denote it by $\mathsf{Ee}$.

In general, the property $\mathsf{Ee}$ may be violated by some small perturbation. 
However, if $f'(0), g'(1) <1$ then we see that  $\mathsf{Ee}$ is a $C^1$-open property.

\subsection{Castration on the overlap}
Finally, we define the castration property.
We put 
$R_f := \mathcal{F}^{-1}(H_g) \subset F_1$ and 
$R_g := \mathcal{G}^{-1}(H_f) \subset G_1$.
We call them ruination region. Note that by construction $R_f$
is the disjoint union of infinitely many intervals accumulating to $f(1)$ and 
$R_g$ to $g(0)$. Then, the castration property is stated as follows:
\begin{center}
$W \subset \mathrm{int}( R_f ) \cup \mathrm{int}(R_g)$. 
\end{center}
We denote this property by $\mathsf{Ca}$.

By the structure of $R_f$ and $R_g$, we observe the following:
\begin{rem}\label{r.hairu}
If $\mathsf{Ca}$ holds, then $g(0) \in \mathrm{int}(R_f)$ and 
$f(1) \in \mathrm{int}(R_g)$.
\end{rem}
A priori, the openness of the property
$\mathsf{Ca}$ is not clear, since it involves the conditions of infinitely many 
intervals.  However, by Remark~\ref{r.hairu}, together with the structure of $R_f$ and $R_g$,
we see that it only involves the conditions of finitely many intervals.
As a corollary, we see that $\mathsf{Ca}$ is  a $C^0$-open property in $\mathcal{A}$.


\subsection{Main statement}
We denote the set of pairs of diffeomorphisms in $\mathcal{A}$ satisfying 
the conditions $\mathsf{So}$, $\mathsf{Ho}$, $\mathsf{Ee}$ and $\mathsf{Ca}$ by $\mathcal{C}$.
Now, we state our main result.
\begin{thm}
If $(f,  g)$ is in $\mathcal{C}$ 
then the (unique) minimal set of $\langle f, g\rangle_{+}$ is homeomorphic to the Cantor set.
\end{thm}

Later we also prove the following:
\begin{prop}\label{p.sonzai}
In $\mathcal{A} \subset (\mathrm{Diff}_{\mathrm{im}}^1(I))^2$,
the set  $\mathcal{C}$ has non-empty interior 
(with respect to the relative topology on $\mathcal{A}$ induced by 
the $C^1$-topology on $(\mathrm{Diff}_{\mathrm{im}}^1(I) )^2$).
\end{prop}

As a corollary, we see that our example can be taken $C^{\infty}$ or $C^{\omega}$.

\begin{rem}\label{r.open}
For $(f, g) \in \mathcal{A}$ satisfying the condition $\mathsf{So}$ and $\mathsf{Ho}$,
we know that the action of $\langle f, g\rangle_{+}$ is not minimal. However, 
in general the author does not know if the minimal set is homeomorphic to the Cantor set or not.
\end{rem}

\section{Proof of the Theorem}

In this section, we give the proof  of our Theorem. 
We fix $(f, g) \in \mathcal{C}$ and consider its unique minimal set $K$.

\subsection{Characterisation of the Cantor set}\label{ss.setumei}
We first remember the famous characterization of the 
Cantor set: If a topological space is a compact, metrizable, perfect and totally disconnected,
then it is homeomorphic to the Cantor set (see for example \cite{W}). 
In our setting, it is clear that $K$ satisfies the first and the second properties. 
The perfectness of the minimal set can be seen easily by definition of the minimal set.
Thus we only need to prove the totally disconnectedness of $K$. 
For that, since we work on one dimensional setting,
we only need to prove the emptiness of the interior.
We state it in more precise way.
\begin{rem}\label{r.reduce} 
For the proof of the Theorem, we only need to prove the following:
For every $x \in K$ and every non-empty open interval $J \subset I$ containing $x$, 
there exists a non-empty interval $L \subset J$ such that $L \cap K = \emptyset$.
\end{rem}

\subsection{Taking the backward image of the orbit}
First, we start from some simple observations.
\begin{lem}\label{l.aruyo}
Let $J \subset I$ be a non-empty open interval.
\begin{itemize}
\item If $J \subset \mathrm{int} (F_1 \setminus W)$ 
(resp. $J \subset \mathrm{int} (G_1 \setminus W)$)
and $J \cap K \neq \emptyset$,
then $\mathcal{F}(J) \cap K \neq \emptyset$ (resp. $\mathcal{G}(J) \cap K \neq \emptyset$).
\item Suppose $J \subset \mathrm{int}(W)$
and $J \cap K \neq \emptyset$.
Then, if $J \subset \mathrm{int}(R_f \setminus R_g)$ (resp. $J \subset \mathrm{int}(R_g \setminus R_f)$),
we have $\mathcal{G}(J) \cap K \neq \emptyset$ (resp. $\mathcal{F}(J) \cap K\neq \emptyset$).
\end{itemize}
\end{lem}

\begin{proof}
We start from the proof of the first item.
By the symmetry, we only consider the case 
$J \subset \mathrm{int} (F_1 \setminus W)$
(and omit the case 
$J \subset \mathrm{int} (G_1 \setminus W)$). 

One important observation for the proof is the following: 
If $p \in F_1 \setminus W$ belongs to $\mathcal{O}_{+}(0)$,
then $\mathcal{F}(p) \in \mathcal{O}_{+}(0)$. Indeed,  in general, for given $p \in \mathcal{O}_{+}(0)$
we have either $f^{-1}(p) \in \mathcal{O}_{+}(0)$ or $g^{-1}(p) \in \mathcal{O}_{+}(0)$ 
(both condition may hold simultaneously).
Since $p \in F_1 \setminus W$, the second option cannot hold. 
So we see that $f^{-1}(p) \in \mathcal{O}_{+}(0)$.
Then, fix $k \geq 1$ so that $f^{-1}(p) \in G_k$ holds. 
By repeating the similar reasoning, 
we can prove that $(g^{-j+1}\circ f^{-1})(p) \in \mathcal{O}_{+}(0)$ for every $j$ with $1 \leq j \leq k$
(by the induction of $j$).
In particular, we have $(g^{-k+1}\circ f^{-1})(p) \in \mathcal{O}_{+}(0)$,
which is equal to $\mathcal{F}(p)$ by definition. 
Thus we see that $\mathcal{F}(J)$ contains a point 
$\mathcal{F}(p) \in \mathcal{O}_{+}(0) \subset K$, which implies $\mathcal{F}(J) \cap K \neq \emptyset$. 

Let us consider the proof of the second item.
Again by the symmetry between $f$ and $g$, we only treat the case 
where $J \subset \mathrm{int}(R_f \setminus R_g)$ and omit the case
$J \subset \mathrm{int}(R_g \setminus R_f)$.
Suppose $J \subset \mathrm{int}(R_f \setminus R_g)$ and 
there exists a point $p \in J \cap \mathcal{O}_{+}(0)$.
Then, as is the previous discussion, one of the following options holds: 
$\mathcal{F}(p)  \in \mathcal{O}_{+}(0)$ or $\mathcal{G}(p)\in \mathcal{O}_{+}(0)$. 
However, the condition $p \in R_f$ implies $\mathcal{F}(p) \in \mathrm{int}(H_f)$ 
and this, together with Proposition~\ref{p.tenganai} prohibits the first option.
Thus we have $\mathcal{G}(p)  \in \mathcal{O}_{+}(0)$,
which implies what we claimed. 
\end{proof}

In the similar way, we have the following (we omit the proof of it). 
\begin{lem}\label{l.naiyo}
If $J \subset \mathrm{int}(R_f \cap R_g)$, then $J \cap K = \emptyset$.
\end{lem}

We prepare one definition.
For $(f, g) \in \mathcal{C}$, we define 
\[
B_f := \partial (H_f \cup (R_f \cap R_g)) \cup \partial F_1, \quad 
B_g := \partial (H_g \cup (R_f \cap R_g)) \cup \partial G_1,
\]
where $\partial X$ denotes the boundary of $X$.
Then we have the following.
\begin{lem}\label{l.sukima}
If $J \subset I$ is an open interval with $J \cap B_f \neq \emptyset$ (resp. $J \cap B_g \neq \emptyset$), then there exists 
a non-empty open interval $U \subset J$ with $U \cap K \neq \emptyset$. 
In other words, 
$B_f \subset K$
(resp. $B_g \subset K$).
\end{lem}

\begin{proof}
Take $J \subset I$ with $J \cap B_f \neq \emptyset$. 
We divide our situation into four cases:  
$J \cap H_f \neq \emptyset$,
$J \cap \partial (R_f \cap R_g) \neq \emptyset$, $f(1) \in J$ or $f^2(1) \in J$.

For the first case, remember Proposition~\ref{p.tenganai}:
It implies that every non-empty open set $U \subset J \cap \mathrm{int}(H_f)$
is disjoint from $K$. Thus one of such $U$ gives us the desired set. 

In the second case, using the castration property 
(that is, $W \subset \mathrm{int}(R_f) \cup \mathrm{int}(R_g)$),
we see that there exists 
a non-empty open set $U \subset J \cap(R_f \cap R_g)$. 
By Lemma~\ref{l.naiyo}, this $U$ gives the desired set.
The third case is reduced to the second case by Remark~\ref{r.hairu} and the fact that 
$R_f$ accumulates to $f(1)$.
 
In the last case, consider $f^{-1}(J)$. Then it contains $f(1)$. 
Thus by the previous argument, 
we can take $U \subset f^{-1}(J)$ with $U \cap K = \emptyset$. Then,
by repeating the similar argument as Lemma~\ref{l.aruyo}, 
we see that $f(U) \subset J$ gives us the desired interval.

\end{proof}

\subsection{Induction: the proof of the Theorem}\label{ss.proof}

Now we start the proof of the Theorem.
Remember that to prove the Theorem we only need to
prove the statement in Remark~\ref{r.reduce}.
The idea of the proof is 
using Lemma~\ref{l.aruyo}, \ref{l.naiyo} and \ref{l.sukima}
inductively to examine the possible past behavior of $J$.

We start the proof for $J$ satisfying $ J \cap (F_1 \cup G_1) \neq \emptyset$.
The other case is easy to prove and will be discussed later.

\begin{proof}[Proof of the Theorem for $J$ with  $ J \cap (F_1 \cup G_1) \neq \emptyset$]

Given $J \subset I $ with $J \cap (F_1 \cup G_1) \neq \emptyset$,
we define a $C^1$-map $\tau: J \to I$ which is a diffeomorphism on the image
and a non-empty open interval $U \subset \tau (J)$.
After that we will see that $L := \tau^{-1}(U)$ gives us the claimed interval 
in Remark~\ref{r.reduce}.
For that, we define a finite sequence of maps $(\tau_n)$ and intervals $(J_n)$ inductively as follows.

We put $\tau_0 := \mathrm{id}$ and $J_0 := J$.
Suppose that we have defined $\tau_{k}$ and $J_{k}$  ($k$ is some non-negative integer). 
Then we proceed as follows: First, if $J_k \cap (B_f \cup B_g)\neq \emptyset$, then take 
the interval $U$ in the conclusion of Lemma~\ref{l.sukima}, put $\tau = \tau_k$ and finish the construction.
Suppose not. 
Then we have five possibilities: $J_k \subset H_f$, $J_k \subset H_g$,
$J_k \subset W$, $J_k \subset F_1 \setminus (H_f \cup W)$
or $J_k \subset G_1 \setminus (H_g \cup W)$.
We proceed as follows. 
\begin{itemize}
\item In the first (resp. second) case,
set $\tau= \tau_{k+1} := \mathcal{F} \circ \tau_k$ (resp. $\tau= \tau_{k+1} := \mathcal{G} \circ \tau_k$),
$U = J_{k+1} := \mathcal{F} (J_k)$, 
(resp. $U = J_{k+1} := \mathcal{G} (J_k)$) and finish the construction.
Note that by Proposition~\ref{p.tenganai}, we see that $U \cap K = \emptyset$. 
\item Suppose $J_k \subset W$. In this case, we know that $J$ is contained in either 
$\mathrm{int}(R_f\setminus R_g)$ or $\mathrm{int}(R_g \setminus R_f)$
(remember that we are under the assumption $J \cap (B_f \cup B_g) = \emptyset$).
\begin{itemize}
\item If  $J \subset \mathrm{int}(R_f\setminus R_g)$, then put $\tau_{k+1} := \mathcal{G} \circ \tau_k$, 
$J_{k+1} := \mathcal{G} (J_k)$ and continue the construction.
\item Otherwise (that is, if  $J \subset \mathrm{int}(R_g\setminus R_f)$), 
then put $\tau_{k+1} := \mathcal{F} \circ \tau_k$,
$J_{k+1} := \mathcal{F} (J_k)$ and continue the construction.
\end{itemize}
\item In the fourth (resp. fifth) case, 
set $\tau_{k+1} := \mathcal{F} \circ \tau_k$ (resp. $\tau_{k+1} := \mathcal{G} \circ \tau_k$),
$J_{k+1} := \mathcal{F} (J_k)$, 
(resp. $J_{k+1} := \mathcal{G} (J_k)$) and continue the construction. 
\end{itemize}

Note that, by eventual expansion property, this process finishes in finite 
steps. Thus for given $J$, we can take $\tau$ and $U \subset \tau(J)$.
Note that, by construction, $\tau|_{J}: J \to I$ is a diffeomorphism on its image and
$U \subset I$ satisfies $U \cap K =\emptyset$.

Then consider the non-empty open interval $\tau^{-1}(U) \subset J$. 
We claim that  $\tau^{-1}(U) \cap K = \emptyset$.
Indeed, suppose $\tau^{-1}(U) \cap K \neq \emptyset$.
This implies $\tau^{-1}(U) \cap \mathcal{O}_{+}(0) \neq \emptyset$. 
However, applying Lemma~\ref{l.aruyo} or \ref{l.naiyo} to each 
option of the construction of $\tau_1$, 
we see that $\tau_1(\tau^{-1}(U))$
contains a point of $\mathcal{O}_{+}(0)$. 
Then by induction 
we conclude that $\tau_j(\tau^{-1}(U)) \cap \mathcal{O}_{+}(0) \neq \emptyset$ for every $j$. 
In particular, we see that $\tau(\tau^{-1}(U)) = U$ contains some point of $K$,
which contradicts to the choice of $U$.
\end{proof}

We give the proof for the case where $J$ is outside $F_1 \cup G_1$,
which completes the proof of the Theorem.

\begin{proof}[Proof of Theorem for $J$ is not contained in $F_1 \cup G_1$]
In this case, we know that $J$ is contained in 
$f(I) \setminus F_1$ or $g(I) \setminus G$. 
By the symmetry, we only need to consider the first case. 

By shrinking $J$ if necessary, we can assume that there exists $N \geq 2$ such that $J \subset F_N$. 
Thus we have $f^{-N+1}(J) \subset F_1$.
Applying the conclusion of the previous case, 
we take a non-empty open set $U \subset f^{-N+1}(J) $ such that 
$U \cap K = \emptyset$. Then, consider $f^{N-1}(U) \subset J$. 
We claim that $f^{N-1}(U) \cap K = \emptyset$ (which implies what we want to prove). 

Indeed, if not, there exists a point $x \in f^{N-1}(U) \cap \mathcal{O}_{+}(0)$.
Then, repeating the argument used in the proof of Lemma~\ref{l.aruyo}, 
together with the single overlapping property $\mathsf{So}$,
we see that $f^{N-1-j }(U) \cap K \neq \emptyset$ for every $0 \leq j \leq N-1$
(by the induction of $j$).
In particular it implies that $U \cap K \neq \emptyset$, but this is a contradiction.
\end{proof}

%
%

\section{Construction of the example}

In this section, we prove Proposition~\ref{p.sonzai}, that is, 
we describe how we construct the example. 

\subsection{Start of the construction and first modification}

We start from the system
\[
f_{\ast}(x) = (1/2)x, \quad g_{\ast}(x) =(1/2)x +1/2.
\]
Then, put $p:=1/3$ and $q:=2/3$  and take  
intervals $J_p$, $J_q$ centered at $p$, $q$ respectively,
with sufficiently small diameter (for example, $|J_p| = |J_q| =1/(100)$ is enough).

We modify $f_{\ast}$ and $g_{\ast}$ to get the example.
Every modification is performed keeping the symmetry 
of the graph with respect to the diagonal: $1-f(1-x) = g(x)$
holds for every $x \in [0, 1]$.
First we modify $f_{\ast}$ and $g_{\ast}$ to $C^1$-maps $f_{0}$ and $g_{0}$ respectively, 
keeping them in $\mathrm{Diff}_{\mathrm{im}}^{1}$ so that 
the following holds:
\begin{itemize}
\item $f_0 =f_{\ast}$ outside $J_p$  and $g_0=g_{\ast}$ outside $J_q$.
\item There are intervals 
$J'_p \subset \mathrm{int}( J_p)$, $J'_q \subset \mathrm{int}( J_q)$ such that
$J'_q \subset \mathrm{int}(g_0(J'_p)) $ and 
$J'_p \subset \mathrm{int}(f_0(J'_q))$.
\end{itemize}
Such a modification can be done just by modifying local behavior around $p$ and $q$  
expanding. 

Then we have the following:
\begin{prop}\label{p.ana}
There exist intervals $H_p \subset \mathrm{int}(J_p)$ and 
$H_q \subset \mathrm{int}(J_q)$ such that
$J'_p \subset \mathrm{int}(H_p)$, 
$J'_q \subset \mathrm{int}(H_q)$ and  
$f_0(H_q) =H_p$,  $g_0(H_q) =H_p$.
\end{prop}

\begin{proof}
By definition, we have
$f_0 \circ g_0 (J_p) \subset \mathrm{int} (J_p)$. 
So, $ H_{p} := \cap_{n \geq 0} (f_0 \circ g_0)^{n} (J_p)$ is a closed interval 
contained in $\mathrm{int}(J_p)$. 
We also know that it cannot be a point, since it contains $J'_p$.
Furthermore, by definition we can see that $ (f_0 \circ g_0)(H_p) = H_p$.
Then $H_q := g_0(H_p)$ and $H_p$ are the desired intervals. 
\end{proof}

\subsection{Second modification}
From $(f_0, g_0)$, we construct a one-parameter family of 
IFSs $(f_{\varepsilon}, g_{\varepsilon})$ as follows.

We fix a small positive number $k$ ($k < 1/(100)$ is enough) and modify $f_0$, $g_0$ to 
$(f_{\varepsilon}, g_{\varepsilon})$, where $\varepsilon$ is some non-negative 
small real number, keeping being continuous and so that the following holds:
\begin{itemize}
\item On $(1-k, 1)$, $f_{\varepsilon}$ is an affine map with slope $(1/2) + \varepsilon$,
\item on $(0, k)$, $g_{\varepsilon}$ is an affine map with slope $(1/2) + \varepsilon$,
\item and keep the dynamics outside intact.
\end{itemize}

Note that there exists small $\delta > 0$ so that for every $\varepsilon \in (0, \delta)$
the pair $(f_{\varepsilon}, g_{\varepsilon})$
satisfies the condition $\mathcal{A}$, single overlapping 
property $\mathsf{So}$ and hole property $\mathsf{Ho}$. 
Hence the notions such as $F_i$, $G_i$, $W$, $R_f$ and $R_g$ make sense 
for $(f_\varepsilon, g_\varepsilon)$.
These objects vary as $\varepsilon$ varies.
To clarify the dependence, we put $\varepsilon$ to them. 
For example, $F_{2}^{\varepsilon}$ denotes the interval $F_2$ 
for $(f_{\varepsilon}, g_{\varepsilon})$. 
Note that if $k$ is chosen sufficiently small then 
$H_p$ and  $H_q$ do not depend on $\varepsilon$. 
Thus we do not put $\varepsilon$ to them.

We are interested in the shape of 
$R_{f_\varepsilon}^\varepsilon$ and $R_{g_\varepsilon}^\varepsilon$ in $W^\varepsilon$.
To compare them of different $\varepsilon$, we introduce the following notion. 
Let $\varepsilon_0, \varepsilon_1 \in (0, \delta)$ be two different numbers. 
Then we define $\phi_{\varepsilon_0, \varepsilon_1}$ to be the unique 
orientation preserving affine homeomorphism from
$W^{\varepsilon_0}$ to $W^{\varepsilon_1}$.

In $W^{\varepsilon}$, we define two sequences of intervals.
For $n \geq 0$, we define $P^{\varepsilon}_n$ as follows:
\[
P^{\varepsilon}_n := \{ x \in W^{\varepsilon} \mid  f_{\varepsilon}^{-n} \circ g_{\varepsilon}^{-1} (x) \in H_p\}. \quad
\]
Similarly, we define $Q^{\varepsilon}_n$ as follows:
\[
Q^{\varepsilon}_n := \{ x \in W^{\varepsilon} \mid  g_{\varepsilon}^{-n} \circ f_{\varepsilon}^{-1} (x) \in H_q\}. \quad
\]
For fixed $\varepsilon$, the interval
$P^{\varepsilon}_n$, $Q^{\varepsilon}_n$ may be empty, but for large $n$ it is not empty. 
Note that $\{ Q^{\varepsilon}_n\}$ accumulates to $f_{\varepsilon}(1)$,
$\{ P^{\varepsilon}_n\}$ to $g_{\varepsilon}(0)$.
We also have
$R_{g_\varepsilon}^{\varepsilon} = \coprod P^{\varepsilon}_n$ and 
$R_{f_\varepsilon}^{\varepsilon} = \coprod Q^{\varepsilon}_n$.

Then we define $C \subset (0, \delta)$ as follows:
\[
C := \{ \varepsilon \in (0, \delta) \mid g_{\varepsilon}(0) \in R^{\varepsilon}_{f_{\varepsilon}},
f_{\varepsilon}(1) \in R^{\varepsilon}_{g_{\varepsilon}} \}.
\]
In the definition of $C$, the condition $f_{\varepsilon}(1) \in R^{\varepsilon}_{g_{\varepsilon}}$ is in fact redundant, 
because of the symmetry.

Now we prove the following (later we will prove a stronger statement):
\begin{prop}
$C$ is not empty. 
\end{prop}

\begin{proof}
Consider the set $H'_p := \cup_{n \geq0} \, g_{\varepsilon}^{n}(H_p)$. 
This is a disjoint union of intervals converging to $1$, and we can see that 
$H'_p$ does not depend on $\varepsilon$.

The condition $\varepsilon \in C$ is equivalent to $f^{-1}_{\varepsilon} (g_{\varepsilon}(0)) \in H'_p$.
So we are interested in how $f^{-1}_{\varepsilon} (g_{\varepsilon}(0))$ varies when $\varepsilon$ varies.
By definition, we can see that for $\varepsilon$ near $0$,
$f_{\varepsilon}^{-1} ( g_{\varepsilon}(0))$ is monotone increasing as $\varepsilon$ decreases
and it converges to $1$ when $\varepsilon \searrow 0$. 
Thus, for sufficiently large $n$, there is an interval $C_n \subset C$ such that for 
$\varepsilon \in C_n$
we have $f_{\varepsilon}^{-1} ( g_{\varepsilon}(0)) \in g_{\varepsilon}^{n}(H_p)$.
In particular, $C$ is not empty.
Finally, note that 
because of the symmetry we do not need to pay attention to $f_{\varepsilon}(1) \in R_f$.
\end{proof}

We fix $\alpha \in C$. Then we claim the following:
\begin{prop}
Let $\alpha \in C$. 
There exists a sequence $(\alpha_n) \subset C$ such that 
$\alpha_n \searrow 0$, $\alpha_0=\alpha$, 
$\phi_{\alpha_{m}, \alpha_n}(R_{f_{\alpha_m}}^{\alpha_m}) =R_{f_{\alpha_n}}^{\alpha_n}$ and 
$\phi_{\alpha_{m}, \alpha_n}(R_{g_{\alpha_m}}^{\alpha_m}) =R_{g_{\alpha_n}}^{\alpha_n}$
for every $m$ and $n$.
\end{prop}

\begin{proof}
Again consider the set $H'_p := \cup_{n \geq0} \, g_{\varepsilon}^{n}(H_p)$. 
By definition, this set has a self-similarity in the following sense:
$g_{\varepsilon}(H'_p) = H_p' \cap g_{\varepsilon}(I) $.

For $n \geq 0$, we define the number $\alpha_n$ such that 
$f^{-1}_{\alpha_n}\circ g_{\alpha_n}(0) =  g_{\alpha_0}^{n}\circ f_{\alpha_0}^{-1}\circ g_{\alpha_0}(0)$ holds. 
For each $n \geq 0$ such $\alpha_n \in C$ exists uniquely because of the 
monotonicity of $f_{\varepsilon}^{-1}\circ g_{\varepsilon}(0)$ with respect to $\varepsilon$.
By construction we can see that $\alpha_n \searrow 0$ as $n \to +\infty$.

Then, for this sequence $(\alpha_n)$, consider the following diagram:
\begin{center}
$ \xymatrix{
\ar[d]_{\phi_{\alpha_m, \alpha_n}} W^{\alpha_m} \ar[rr]_{ f^{-1}_{\alpha_m} \hspace{20pt}}  & &
(f_{\alpha_m}^{-1} \circ g_{\alpha_m}(0), 1) \ar[d]_{g_{\alpha_0}^{(n-m)}} & &
\ar@{_{(}->}[ll] H'_{p} \ar@{^{(}->}[lld] \ar[d]^{g_{\alpha_0}^{(n-m)}} \\
W^{\alpha_n} \ar[rr]^{f^{-1}_{\alpha_n} \hspace{20pt}}  & &
(f_{\alpha_n}^{-1} \circ g_{\alpha_m}(0), 1)  & &
\ar@{_{(}->}[ll] H'_{p}}$
\end{center}
Together with the self similarity of $H'_p$ and the equality
$R^{\alpha_n}_{g_{\alpha_n}} = f_{\alpha_n}(H'_p)  \cap W^{\alpha_n}$ for every $n \geq 0$, 
we see $\phi_{\alpha_m, \alpha_n}(R_{g_{\alpha_m}}^{\alpha_m}) =R_{g_{\alpha_m}}^{\alpha_m}$ holds
for every $m$ and $n$.
By the symmetry, we also have the equality for $R_{f_{\alpha_n}}^{\alpha_n}$.
\end{proof}

Now we are ready for the proof.
\begin{proof}[End of the proof of Proposition~\ref{p.sonzai}]
Let $(f_{\varepsilon}, g_{\varepsilon})$ as above. 
We fix the sequence $(\alpha_n)$. 
Then we take a $C^1$-diffeomorphism $\gamma :W^{\alpha_0} \to W^{\alpha_0}$
so that 
$W \subset \mathrm{int}(R_{f_{\alpha_0}}^{\alpha_0} \cup \gamma (R_{g_{\alpha_0}}^{\alpha_0}))$ holds. 
The existence of such $\gamma$ is easy to see using the fact that $\alpha_0 \in C$. 

Then, for each $n$ 
we modify $g_{\alpha_n}$ to $g_{\alpha_n, \bullet}$ so that for  $ x\in W^{\alpha_n}$, 
\[
(g_{\alpha_n, \bullet})^{-1}(x) =  
(g_{\alpha_n})^{-1} \circ ( \phi_{\alpha_n, \alpha_0})^{-1}\circ \gamma^{-1} \circ (\phi_{\alpha_n, \alpha_0})(x)
\]
holds, and keep intact outside. Then, this $(f_{\alpha_n}, g_{\alpha_n, \bullet})$ satisfies the 
single overlap, hole and castration property. 

We claim that for sufficiently large $n$, $(f_{\alpha_n}, g_{\alpha_n, \bullet})$ also satisfies the 
eventual expansion property $\mathsf{Ee}$. We are interested in the differential of the first return map. 
Note that, by increasing $n$, the number of the iteration for the first return increases, 
while the differential $(\phi_{\alpha_n, \alpha_1})^{-1}\circ \gamma \circ  \phi_{\alpha_n, \alpha_1}$ 
is intact. The differetial of $g_{\alpha_n}$ may vary, but it is a bounded change. 
Hence by taking sufficiently large $n$, we get the uniform expansion of the first return map
on $W^{\alpha_n}$.

The openness of the properties $\mathsf{So}$,  $\mathsf{Ca}$ and $\mathsf{Ee}$ 
are easy to see. 
The $C^1$-openness of $\mathsf{Ho}$ comes from the fact 
that the hypothesis in Proposition~\ref{p.ana} is $C^1$ (indeed, $C^0$) open property.
\end{proof}

\appendix 

\section{Another example}
After finding the previous example, Masayuki Asaoka
showed me another type of example. In this appendix we present it.

We fix a small number $\varepsilon> 0$ 
(in practice, $\varepsilon  = 1/(100)$ is enough) and 
another small number $\lambda >0 $ smaller than $1/2$. 

We define as follows:
\begin{center}
$I_{-1}:= [0, 1/3 -\varepsilon],\quad
I_{ 0 }:= [1/3 +\varepsilon, 2/3 -\varepsilon], \quad
I_{1}:= [2/3 +\varepsilon,  1]$.
\end{center}
Then take $(f, g) \in \mathcal{A}$ such that the following holds:

\begin{itemize}
\item $f(I_{-1}) \subset I_{-1}$, \quad $f(I_{0}) \subset \mathrm{int}(I_{-1})$,
\quad $f(I_{1}) \subset \mathrm{int}(I_{0})$.
\item $g(I_{1}) \subset I_{1}$, 
\quad $g(I_{0}) \subset \mathrm{int}(I_{1})$,
\quad $g(I_{-1}) \subset \mathrm{int}(I_{0})$.
\item $f'|_{I_{-1}\cup I_{0}\cup I_{1}}, g'|_{I_{-1}\cup I_{0}\cup I_{1} }< \lambda$.
\end{itemize}
We call the first and the second conditions {\it inclusion property}, 
and the last condition {\it strong uniform contraction property}.

The construction of such pair of maps in $\mathcal{A}$ is not such difficult 
(see the graph in Figure~\ref{fig.aliappe}), so we skip the detail. Note that these properties are 
open in $\mathcal{A}$ (the conditions $f(I_{-1}) \subset I_{-1}, g(I_{1}) \subset I_{1}$
are not open in $(\mathrm{Diff}^1_{\mathrm{im}})^2$, but in $\mathcal{A}$ they are).

\begin{figure}
\centering
\includegraphics{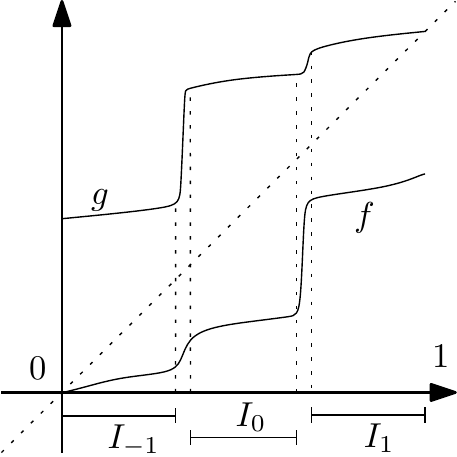}
 \caption{An example of IFS in the Appendix.}
\label{fig.aliappe}
\end{figure}

We claim the following:
\begin{prop}
The (unique) minimal set $K$ for above $(f, g)$ is homeomorphic to the Cantor set.
\end{prop}
For the proof, remember that
we only need to prove the emptiness of the interior of the minimal set (see Section~\ref{ss.setumei}).
For that we prove that the Lebesgue measure of $K$ is zero. We denote the 
Lebesgue measure on $I$ by $\mu$.

First, we define the sequence of compact sets $(\Lambda_n)$ as follows:
\begin{itemize}
\item $\Lambda_0 := I_{-1} \cup I_{0} \cup I_{-1}$.
\item $\Lambda_{k+1} := f(\Lambda_k) \cup g(\Lambda_k)$.
\end{itemize}
By the inclusion property of $f$ and $g$, we know that $(\Lambda_n)$
is a nested sequence of compact sets. Thus $\Lambda := \cap \Lambda_n$
is a non-empty compact set. Furthermore, by definition we see that $\mathcal{O}_{+}(0) \subset \Lambda$.
Thus we have $K \subset \Lambda$ (indeed, we can prove the equality between $K$ and $\Lambda$ but 
since we do not need it we omit it). 

Then we claim the following:
\begin{clm}
$\mu(\Lambda) =0$, in particular $\mu(K) =0$.
\end{clm}
\begin{proof}
We prove that for every $n \geq 0$ we have $\mu(\Lambda_n) \leq (2\lambda)^{n}$, which concludes 
$\mu(\Lambda)=0$ for $\lambda < 1/2$.

The case $n=0$ is clear. Suppose the inequality is true for $n=k$ (where $k$ is some non-negative integer).
First, by definition we have 
\[
\mu(\Lambda_{k+1}) \leq \mu(f(\Lambda_{k})) + \mu(g(\Lambda_{k})).
\]
Then, by uniform contracting property we have 
\[
\mu(f(\Lambda_{k})) < \lambda \mu(\Lambda_{k}), \quad 
\mu(g(\Lambda_{k})) < \lambda \mu(\Lambda_{k}). \]

Thus we have 
\[
\mu(\Lambda_{k+1}) \leq (2\lambda)\mu(\Lambda_{k}) < (2\lambda)^{k+1},
\]
which completes the proof.
\end{proof}

\begin{rem}
The first example and the second one are conceptually different examples.
The first example is constructed by investigating the past behavior of the dynamics, 
while the second one is constructed by investigating the future behavior.
\end{rem}

\end{document}